\documentclass[11pt, draft]{amsart}
\usepackage{amssymb, amstext, amscd, amsmath, amssymb, tikz, amsfonts}
\usepackage{mathtools, xy, paralist, color, dsfont, rotating}
\usepackage{verbatim}
\usepackage{enumerate}
\usepackage{caption}

\usepackage[a4paper]{geometry}
\usepackage{charter}

\numberwithin{equation}{section}

\makeatletter
\def\@cite#1#2{{\m@th\upshape\bfseries%
[{#1\if@tempswa{\m@th\upshape\mdseries, #2}\fi}]}}
\makeatother
\theoremstyle{plain}
\newtheorem{theorem}{Theorem}[section]
\newtheorem{corollary}[theorem]{Corollary}
\newtheorem{proposition}[theorem]{Proposition}

\theoremstyle{definition}
\newtheorem{definition}[theorem]{Definition}

\newtheorem{remark}[theorem]{Remark}

\theoremstyle{remark}

\mathtoolsset{centercolon}
  \newcommand{\A}{{\mathcal{A}}}
  \newcommand{\B}{{\mathcal{B}}}

  \newcommand{\G}{{\mathcal{G}}}
\renewcommand{\H}{{\mathcal{H}}}
  \newcommand{\I}{{\mathcal{I}}}
  
  \newcommand{\K}{{\mathcal{K}}}

\renewcommand{\S}{{\mathcal{S}}}

\newcommand{\eps}{\varepsilon}

\def\ze{\zeta}

\def\la{\lambda}

\def\si{\sigma}

\newcommand{\bA}{\mathbb{A}}

\newcommand{\bC}{\mathbb{C}}
\newcommand{\bD}{\mathbb{D}}
\newcommand{\bF}{\mathbb{F}}
\newcommand{\bN}{\mathbb{N}}

\newcommand{\bR}{\mathbb{R}}
\newcommand{\bZ}{\mathbb{Z}}

\newcommand{\fM}{{\mathfrak{M}}}

\newcommand{\fS}{{\mathfrak{S}}}

\newcommand{\Bi}{{\mathbf{i}}}

\newcommand{\AND}{\text{ and }}

\newcommand{\foral}{\text{ for all }}
\newcommand{\qand}{\quad\text{and}\quad}

\newcommand{\qfor}{\quad\text{for}\ }

\newcommand{\ca}{\mathrm{C}^*}

\newcommand{\cenv}{\mathrm{C}^*_{\textup{env}}}

\newcommand{\ol}{\overline}
\newcommand{\wt}{\widetilde}
\newcommand{\wh}{\widehat}

\newcommand{\ad}{\operatorname{ad}}

\newcommand{\Aut}{\operatorname{Aut}}
\newcommand{\alg}{\operatorname{alg}}

\newcommand{\ev}{\operatorname{ev}}
\newcommand{\GL}{\operatorname{GL}}

\newcommand{\id}{{\operatorname{id}}}

\newcommand{\sn}{\operatorname{sn}}

\newcommand{\sca}[1]{\left\langle#1\right\rangle} 
\newcommand{\bo}[1]{\mathbf{#1}} 

\addtocontents{toc}{\protect\setcounter{tocdepth}{1}}

\begin{document}

\title[Operator Algebras of Higher Numerical Semigroups]{Operator Algebras of Higher Rank Numerical Semigroups}

\author[E.T.A. Kakariadis]{Evgenios T.A. Kakariadis}
\address{School of Mathematics, Statistics and Physics\\ Newcastle University\\ Newcastle upon Tyne\\ NE1 7RU\\ UK}
\email{evgenios.kakariadis@ncl.ac.uk}

\author[E.G. Katsoulis]{Elias~G.~Katsoulis}
\address{Department of Mathematics\\ East Carolina University\\ Greenville\\ NC 27858\\USA}
\email{katsoulise@ecu.edu}

\author[X. Li]{Xin Li}
\address{School of Mathematics and Statistics\\ University of Glasgow\\ University Place\\ Glasgow\\ G12 8QQ\\ UK}
\email{xin.li@glasgow.ac.uk}



\begin{abstract}
A higher rank numerical semigroup is a positive cone whose seminormalization is isomorphic to the free abelian semigroup.
The corresponding nonselfadjoint semigroup algebras are known to provide examples that answer Arveson's Dilation Problem to the negative.
Here we show that these algebras share the polydisc as the character space in a canonical way. 
We subsequently use this feature in order to identify higher rank numerical semigroups from the corresponding nonselfadjoint algebras.
\end{abstract}

\maketitle

\thispagestyle{empty}

\section{Introduction}

Semigroup C*-algebras, i.e., C*-algebras generated by left regular representations of left-cancellative semigroups, form a natural class of C*-algebras which generalize (reduced) group C*-algebras. 
They have been studied for various classes of semigroups, for instance positive cones in totally ordered groups \cite{Co,Dou,Mur}, examples coming from group theory \cite{CL02,CL07,LOS} or examples of number-theoretic origin \cite{CDL,CEL1,CEL2}. 
We refer the reader to \cite{CELY} and the references therein.

Nonselfadjoint semigroup operator algebras are formed by considering the non-invo\-lu\-tive part of the left regular representation, or through families of representations.
They are trivially examples of semicrossed products, a construct introduced by Arveson \cite{Arv67} and formalized by Peters \cite{Pet}, that captures the properties of semigroup actions on C*-algebras.
Actions over $\bZ_+^d$ and $\bF_+^d$ (the free semigroup on $d$ generators) are by now well-studied, with a comprehensive list of pertinent papers being too long to present here.
We direct the interested reader to the surveys \cite{DFK18, DKsurv, Kat1} for more information.
However less is known for other semigroups, even at the level of the semigroup algebras, with only recent dilation results obtained for lattice-ordered semigroups \cite{DFK17, Li15}.

The realm of semigroups is too vast to be treated in one stroke and it is appropriate to reflect on a case-by-case study.
In this paper we focus on a particular class, which we call higher rank numerical semigroups as they generalize classical numerical semigroups in a natural way. 
These are the positive cones $\S$ (always viewed inside a group $\G \simeq \bZ^d$) whose \emph{seminormalization} is isomorphic to $\bZ_+^d$, namely
\[
\S_{\sn} := \{g \in \G \mid ng \in \S \textup{ eventually for } n \in \bN\} \simeq \bZ_+^d.
\]
When $d=1$ this amounts to subsemigroups of the natural numbers with finite complement. 
The motivation to study such structures is two-fold.
On one hand they play an important role in dilation theory; Dritschel-Jury-McCullough \cite{DJM16} have used the algebra associated with $\{0,2,3, \dots\}$ to provide a negative answer to Arveson's Dilation Problem, which asks whether the contractive representations of an operator algebra are automatically completely contractive.
On the other hand, higher rank numerical semigroups are determined by the characters of the algebras: in Proposition~\ref{P:inj theta} we show that these are the only semigroups for which the restriction on point evaluations identifies the character space with $\ol{\bD^d}$.

Our objective in this paper is to show that this setting is also rigid.
In Theorem~\ref{T:hr num} we show that two higher rank numerical semigroups are isomorphic if and only if their nonselfadjoint operator algebras are isomorphic.
In fact we show that being completely isometrically isomorphic coincides with being algebraically isomorphic. 
It is somehow surprising that we make an essential use of the homeomorphism between the character spaces, although they are all the same.
The meta-mathematical note here is that this is the widest class of positive cones this approach tackles.

This rigidity of nonselfadjoint operator algebras is another example of stark contrast with the C*-algebra setting.
The difference is particularly striking for the classical numerical semigroups.
In this case, it is straightforward to check that they all have isomorphic semigroup C*-algebras. 
However, the nonselfadjoint operator algebra remembers the numerical semigroup completely.

\section{Positive Cones}

We give some preliminary results for positive cones.
We will write $\bN = \{1, 2, \dots\}$ and $\bZ_+ = \{0, 1, 2, \dots\}$.
The reader may refer to \cite{Pau02} for the general theory of nonselfadjoint operator algebras and dilations of their representations, that we will avoid repeating here.
We just recall that the C*-envelope $\cenv(\A)$ of a nonselfadjoint operator algebra $\A$ is the co-universal C*-algebra in the sense that: (i) there exists a completely isometric homomorphism $i \colon \A \to \cenv(\A) = \ca(i(\A))$; and (ii) for any other completely isometric homomorphism $j \colon \A \to \ca(j(\A))$ there exists a unique $*$-epimorphism $\Phi \colon \ca(j(\A)) \to \cenv(\A)$ such that $\Phi \circ j = i$.
It follows by \cite{DM05} that $\cenv(\A) = \ca(\rho(\A))$ for any completely isometric representation $\rho$ of $\A$ that does not admit non-trivial contractive dilations.

Recall that a \textit{positive cone $\S$} of an abelian group $\G$ is a unital sub-semigroup of $\G$ such that: (i) $\S \cap (-\S) = (0)$; and (ii) for every $g \in \G$ there exist $s, t \in \S$ such that $g = s - t$.
The \textit{Fock representation} $V \colon \S \to \B(\ell^2(\S))$ is given by
\[
V_s e_t = e_{s + t}.
\]
We define
\[
\A(\S) := \ol{\alg}\{V_s \mid s \in \S\} \qand \ca(\S) := \ca(V_s \mid s \in \S).
\]
Since $V_{s} V_{t} = V_{s + t}$ we get that $\A(\S)$ is in fact densely spanned by the monomials $V_s$.
In particular the polynomials in $\A(\S)$ have a unique expression.
Indeed in order to reach a contradiction assume that $\sum_{s \in F} \la_s V_s = 0$ for a finite set $F \subseteq \S$ and for $0 \neq \la_s \in \bC$.
Let $s_0$ be a minimal element in $F$ with respect to the partial order that $\S$ induces in $\G$.
Then taking the compression to the $(s_0,0)$-entry gives the contradiction
\[
\la_{s_0} = \sum_{s \in F} \sca{\la_s V_s e_0, e_{s_0}} = 0.
\]

Notice that the Fock representation does not use any property of the positive cone and can be defined for general left cancellative semigroups, including groups of course.
To allow comparisons however we reserve the notation $U \colon \G \to \B(\ell^2(\G))$ for the left regular group representation of an abelian group $\G$.
In this case $\ca(\G)$ is the usual (reduced) group C*-algebra.

The following proposition can be derived as an application of dilation results for C*-dynamics, that have appeared for example in \cite{DFK17, DP15}, when applied for trivial dynamical systems.
In the absence of the dynamics, a simpler proof can be given, and it is included here for completeness.

\begin{proposition}\label{P:is in G}
Let $\S$ be a positive cone of an abelian group $\G$.
Then the mapping $V_s \mapsto U_s$ extends to a completely isometric map $\rho \colon \A(\S) \to \ca(\G)$.
\end{proposition}

\begin{proof}
Since polynomials in $\A(\S)$ have a unique expression the map $V_s \mapsto U_s$ admits a unique linear extension.
It suffices to show that this extension is isometric, i.e.,
\[
\| \sum_{s \in F} \la_s V_s \| = \| \sum_{s \in F} \la_s U_s \| \quad \text{for all finite } F \subset \S.
\]
Similar arguments at any matrix level yield that this mapping is completely isometric, and thus it extends to a completely isometric map $\rho \colon \A(\S) \to \ca(\G)$.

By identifying $\ell^2(\S)$ with the obvious subspace inside $\ell^2(\G)$ we get that
\begin{align*}
\| \sum_{s \in F} \la_s V_s \|
& =
\| P_{\ell^2(\S)} \bigg(\sum_{s \in F} \la_s U_s \bigg) |_{\ell^2(\S)} \|
\leq
\| \sum_{s \in \S} \la_s U_s \|.
\end{align*}
For the reverse inequality fix $\eps> 0$.
Let $\xi = \sum_{i=1}^n k_i e_{g_i}$ in the unit ball of $\ell^2(\G)$ such that
\[
\| \sum_{s \in F} \la_s U_s \| - \eps \leq  \| \sum_{s \in F} \la_s U_s \xi \|_{\ell^2(\G)}.
\]
Since $\S$ is a positive cone we have that there are $s_i ,t_i \in \S$ such that $g_i = s_i - t_i \foral i=1, \dots, n$.
Set $t := \sum_{i=1}^n t_i \in \S$ so that $t + g_i \subset \S$ for all $i=1, \dots, n$.
Then the vector
\[
\xi' := U_t \xi = \sum_{i=1}^n k_i e_{t + g_i}
\]
is in the unit ball of $\ell^2(\S)$.
Therefore we obtain
\begin{align*}
\| \sum_{s \in F} \la_s U_s \| - \eps
& \leq
\| U_t \sum_{s \in F} \la_s U_s \xi \|_{\ell^2(\G)}
= 
\| \sum_{s \in F} \la_s U_s U_t \xi \|_{\ell^2(\G)} \\
& = 
\| \sum_{s \in F} \la_s U_s \xi' \|_{\ell^2(\G)} 
=
\| \sum_{s \in F} \la_s V_s \xi' \|_{\ell^2(\S)}
\leq
\| \sum_{s \in F} \la_s V_s \|.
\end{align*}
As $\eps > 0$ was arbitrary we have equality of the norms.
\end{proof}

\begin{corollary}\label{C:cenv}
Let $\S$ be a positive cone of an abelian group $\G$.
Then $\ca(\G)$ is the C*-envelope of the semigroup algebra $\A(\S)$.
\end{corollary}

\begin{proof}
By Proposition \ref{P:is in G} we have that $\A(\S) \hookrightarrow \ca(\G)$ completely isometrically.
As the copy of $\A(\S)$ contains the generators of $\ca(\G)$ we get that $\ca(\G)$ is a C*-cover of $\A(\G)$ and it is generated by unitaries.
Since a contractive dilation of a unitary is trivial we get that this is a maximal representation of $\A(\S)$ and thus $\ca(\G)$ is the C*-envelope \cite{DM05}.
\end{proof}

The completely contractive representations of $\A(\S)$ are characterized in the following theorem.

\begin{theorem}
Let $\S$ be a positive cone of an abelian group $\G$.
A representation $\rho \colon \A(\S) \to \B(\H)$ is completely contractive if and only if there is a unitary representation $U \colon \G \to \B(\K)$ for $\K \supseteq \H$ such that $\rho(V_s) = P_\H U_s |_\H$ for all $s \in \S$.
\end{theorem}

\begin{proof}
By Proposition \ref{P:is in G} we have that $\A(\S) \subset \ca(\G)$.
If $\rho$ is a completely contractive representation of $\A(\S)$ then it extends to a completely contractive map of $\ca(\G)$ and Stinespring's Theorem produces the required unitary representation.
Conversely, if we have such a unitary representation then its compression to $\H$ is a completely contractive map of $\ca(\G)$.
Hence the restriction $\rho$ to $\A(\S)$ is a completely contractive map.
\end{proof}

\begin{remark}
There is a small subtlety in the above, even for $d=1$.
One may have expected that starting with a contraction $T$ we can build a representation of $\A(\S)$ by assigning $V_s \mapsto T^s$.
However this does not hold.
In \cite{DJM16} it is shown that there is a contraction that fails to induce a completely contractive representation for the semigroup $\S = \{0, 2,3, \dots\}$.
The example given fails to be $2$-contractive.
\end{remark}

Henceforth we will restrict our attention to $\G = \bZ^d$ for some finite $d$.
For a positive cone $\S \subset \bZ_+^d$ of $\bZ^d$ we make the following identifications
\[
\S \setminus \{0\} \ni s \equiv V_s \equiv z^s \in \bA(\bD^d)
\qand
V_0 \equiv z^0 = I.
\]
This has two consequences.
First, by Corollary \ref{C:cenv}, we have a canonical identification $\A(\S) \subset \A(\bZ_+^d) \subset \ca(\bZ^d)$, and so we can use the Fourier transform on $\A(\S)$ that is inherited from $\ca(\bZ^d)$.
Secondly, by identifying an element $x \in \A(\S)$ with an analytic function $f_x \in \bA(\bD^d)$ we see that the $s$-Fourier co-efficient of $x$ coincides with $f_x^{(s)}(0)$.
From now on we will write $f$ for both $f_x$ and $x$.
A straightforward application gives the following corollary.

\begin{corollary}\label{C:order}
Let $\S \subset \bZ_+^d$ be a positive cone in $\bZ^d$.
For $f \in \A(\bZ_+^d)$ we have that $f \in \A(\S)$ if and only if $\{s \in \bZ_+^d \mid f^{(s)}(0) \neq 0\} \subset \S$.
Therefore $s \in \S$ if and only if there exists an $f \in \A(\S)$ such that $f^{(s)}(0) \neq 0$.
\end{corollary}

\begin{corollary}
Let $\S_1 \subset \bZ_+^{d_1}$ and $\S_2 \subset \bZ_+^{d_2}$ be positive cones in $\bZ^{d_1}$ and $\bZ^{d_2}$ respectively.
If $\A(\S_1)$ and $\A(S_2)$ are completely isometrically isomorphic then $d_1 = d_2$.
\end{corollary}

\begin{proof}
Being completely isometrically isomorphic yields that the associated C*-enve\-lopes $\ca(\bZ^{d_1})$ and $\ca(\bZ^{d_2})$ are $*$-isomorphic.
\end{proof}

If $\rho \colon \A \to \B$ is an algebraic epimorphism for two Banach algebras $\A$ and $\B$, then the discontinuity of $\rho$ is quantified by the ideal
\[
\fS(\rho) : = \{b \in \B \mid \exists (a_n) \subset \A \text{ such that } a_n \to 0 \text{ and } \rho(a_n) \to b \}.
\]
By the closed graph theorem $\rho$ is continuous if and only if $\fS(\rho) = (0)$.
Due to a result of Sinclair \cite{Sin75}, for any sequence $(b_n)$ in $\B$ there exists an $N \in \bN$ such that
\[
\ol{b_1 \cdots b_N \fS(\rho)} = \ol{b_1 \cdots b_n \fS(\rho)}
\; \text{and} \;
\ol{\fS(\rho) b_n \cdots b_1} = \ol{\fS(\rho) b_{N} \cdots b_1},
\foral n \geq N.
\]

\begin{proposition}\label{P:aut cont}
Let $\S \subset \bZ_+^d$ be a positive cone in $\bZ^d$.
Then any algebraic epimorphism $\rho \colon \A \to \A(\S)$ for any Banach algebra $\A$ is automatically continuous.
\end{proposition}

\begin{proof}
Fix $s \in \S$.
By applying Sinclair's result \cite{Sin75} for $b_n = (V_s)^n$ we get that there exists an $N \in \bN$ such that
\[
\ol{(V_s)^N \fS(\rho)} = \ol{(V_s)^{N+n} \fS(\rho)} \foral n \in \bN.
\]
As $V_s$ is an isometry we have that $\fS(\rho) = \ol{(V_s)^n \fS(\rho)}$ for all $n \in \bN$.
However the Fourier transform yields $\bigcap_{n \in \bN} \ol{(V_s)^n \I} = (0)$ for any ideal $\I \subset \A(\S)$.
Indeed if $0 \neq f \in \bigcap_{n \in \bN} \ol{(V_s)^n \I}$ then for every $n$ there would be an analytic polynomial $g_n \in \I$ so that $f = z^{ns} g_n$.
If $s_0$ is minimal so that $f^{(s_0)}(0) \neq 0$ and $F \subseteq \S$ is the set of minimal $r \in \S$ so that $g^{(r)}(0) \neq 0$ for some $g \in \I$, then we get that $s_0 \in ns + F$ for all $n \in \bN$.
In particular there are $r, r' \in F$ so that $s + r = 2s + r'$ and so $r = s + r'$.
This contradicts minimality of $r$ as $s + r'$ comes from $z^sg$ for some $g \in \I$ with $g^{(r')}(0) \neq 0$.
Applying for $\I = \fS(\rho)$ gives the required $\fS(\rho) = \bigcap_{n \in \bN} \ol{(V_s)^n \fS(\rho)} = (0)$.
\end{proof}

\section{Numerical semigroups}\label{sec:NumSgp}

Recall that a positive cone $\S$ of a group $\G$ is called \textit{seminormal} if whenever $3s = 2t$ for $s, t \in \S$ then there exists a (necessarily unique) $p \in \S$ such that $2p = s$ and $3p = t$ \cite[Definition 1.7]{CHWW14}; equivalently if $p=t-s$ is in $\S$.
Every positive cone admits a \textit{seminormalization} which by \cite[Example 1.12]{CHWW14} can be expressed as
\[
\S_{\sn}:= \{g \in \G \mid ng \in \S \textup{ eventually for } n \in \bN\}.
\]
Alternatively $\S_{\sn}$ is the universal seminormal monoid that contains an injective copy of $\S$ \cite[Lemma 1.11]{CHWW14}.
The seminormalization of a positive cone is itself a positive cone (for the same generating group).

\begin{remark}
Positive cones in $\bZ$ are also known as {numerical semigroups} and have several equivalent characterizations.
For example $\S$ is a numerical semigroup, if and only if $\gcd(\S) = 1$ if and only if there is an $N \in \S$ such that $n \in S$ for all $n > N$, if and only $\S_{\sn} = \bZ_+$.
We will consider their higher rank analogue.
\end{remark}

\begin{definition}
A positive cone $\S$ of a group $\G$ is called a \textit{higher rank numerical semigroup} if $\S_{\sn} \simeq \bZ_+^d$.
If $d=1$ then $\S$ is called simply a \textit{numerical semigroup}.
\end{definition}

\begin{remark}\label{R:rem}
The above definition implies several items.
First of all it is not hard to see that an isomorphism $\si \colon \S_1 \to \S_2$ between two positive cones induces an isomorphism $\wt{\si} \colon \G_1 \to \G_2$ of their generating groups given by $\wt{\si}(s - t) = \si(s) - \si(t)$.
Moreover $\wt{\si}$ restrics to an isomorphism of the seminormalizations.
Therefore if $\S \subset \G$ is a higher rank numerical semigroup then $\S \hookrightarrow \bZ_+^d$ and $\G \simeq \bZ^d$.
\end{remark}

Let us restrict for a moment to positive cones of $\bZ^d$ with seminormalization \textit{equal} to $\bZ_+^d$.
For notational purposes we write $\{\bo{1}, \dots, \bo{d}\}$ for the usual generators in $\bZ^d$.
We will also use the multivariable notation
\[
z^s = z^{s_1} \cdots z^{s_d} 
\qfor z = (z_1, \dots, z_d) \in \bC^d
\AND s = (s_1, \dots, s_d) \in \bZ^d.
\]
Let $\S \subset \bZ_+^d$ be a positive cone of $\bZ^d$.
The representation of Proposition \ref{P:is in G} allows to see $\A(\S)$ inside $\A(\bZ_+^d)$, and so there is a continuous map between their character spaces $\ol{\bD}^d$ and $\fM_\S$; namely
\[
\iota^* \colon \ol{\bD}^d \to \fM_\S : \ze \mapsto \ev_\ze|_{\A(\S)}.
\]
The next proposition shows that this map is injective exactly when $\S_{\sn} = \bZ_+^d$.

\begin{proposition}\label{P:inj theta}
Let $\S \subset \bZ_+^d$ be a positive cone of $\bZ^d$.
Let $\iota^* \colon \ol{\bD}^d \to \fM_\S$ be the continuous map induced by the embedding $\A(\S) \hookrightarrow \A(\bZ_+^d)$.
Then the following are equivalent:
\begin{enumerate}
\item $\S_{\sn} = \bZ_+^d$;
\item the intersection of $\S$ with any axis is a non-trivial positive cone of $\bZ$;
\item $\iota^*$ is injective.
\end{enumerate}
In particular, $\iota^*$ is a homeomorphism when it is injective.
\end{proposition}

\begin{proof}

[(i) $\Leftrightarrow$ (ii)]:
For simplicity let us write $\S(i) : = \S \cap \{ n \Bi \mid n \in \bZ_+\}$.
If $\bZ_+^d = \S_{\sn}$ then $n \Bi \in \S$ eventually for $n \in \bN$.
Hence $\gcd(\S(i)) = 1$ giving that $\S(i)$ is a positive cone in $\bZ$.
Conversely if $\S(i)$ is a positive cone then
\[
\Bi \in \bN \Bi \subset (\S(i))_{\sn} \subset \S_{\sn},
\]
and thus $\bZ_+^d = \S_{\sn}$.

\smallskip

\noindent
[(ii) $\Leftrightarrow$ (iii)]:
Suppose that $\iota^*$ is injective.
First we show that $\S$ intersects with all axes.
Assume without loss of generality that $\S(1) = \{0\}$.
Then for every $s \in \S$ we have that $s_1 = 0$, and so
\[
\ev_{(\la, 0, \dots, 0)}(z^{s}) = 0 = \ev_{(0,\dots,0)}(z^{s})
\]
for any $\la \neq 0$.
Hence $\ev_{(\la, 0, \dots, 0)} = \ev_{(0,\dots,0)}$ for any $\la \neq 0$, which contradicts injectivity of $\iota^*$.
Secondly we show that every $\S(i)$ is a positive cone in $\bZ$.
Without loss of generality assume that $\S(1)$ is not such and set $k : = \gcd(S(1)) \neq 1$.
Let $\la, \mu$ be two distinct non-trivial $k$-th roots of the unit.
If $s \in \S(1)$ then
\[
\ev_{(\la, 0, \dots, 0)}(z^s) = \la^{s_1} = \mu^{s_1} = \ev_{(\mu, 0, \dots, 0)}(z^s).
\]
If $s \in \S \setminus \S(1)$ then there is at least one $j \in \{2, \dots, d\}$ such that $s_j \neq 0$ and so
\[
\ev_{(\la, 0, \dots, 0)}(z^{s}) = 0 = \ev_{(\mu, 0, \dots, 0)}(z^s).
\]
Therefore $\ev_{(\la, 0, \dots, 0)} = \ev_{(\mu, 0, \dots, 0)}$ which again contradicts injectivity of $\iota^*$.

For the converse recall that a semicharacter on $\S$ is a semigroup homomorphism $\chi \colon \S \to \ol{\bD}$. 
By \cite[Theorem 4.2.1]{GT06} the character space of $\A(\S)$ is homeomorphic to the semicharacter space of $\S$.
We will show that every semicharacter of $\S$ extends uniquely to a semicharacter of $\S_{\sn}$.
This will give that the character spaces of $\A(\S)$ and $\A(\S_{\sn})$ are homeomorphic, and so if $\S_{\sn} = \bZ_+^d$ then $\iota^*$ is a homeomorphism.
To this end for $\chi \colon \S \to \ol{\bD}$ we define $\wt{\chi} \colon \S_{\sn} \to \ol{\bD}$ by
\[
\wt{\chi}(t) :=
\begin{cases}
\chi((n+1)t) / \chi(nt) & \textup{ if } (n+1)t, nt \in \S \textup{ and } \chi(nt) \neq 0 \textup{ for some } n \in \bN, \\
0 & \textup{ if } \chi(nt) = 0 \textup{ for every } n \in \bN \textup{ with } nt \in \S.
\end{cases}
\]
To see that $\wt{\chi}$ is well defined first suppose that $\chi(nt) = 0$ for some $n \in \bN$.
Then for every $m \in \bN$ with $m t \in \S$ we have that 
\[
\chi(mt)^n = \chi(m n t) = \chi(nt)^m = 0,
\]
and so $\chi(mt) = 0$.
Now by definition for every $t \in \S_{\sn}$ there exists an $n \in \bN$ such that $nt \in \S$ and $(n+1) t \in \S$.
If there are distinct $n, m \in \bN$ such that $(n+1)t, nt \in \S$ and $(m+1)t, mt \in \S$ then we have
\[
\chi((n+1)t) \chi(mt) = \chi((n+m+1)t) = \chi((m+1)t) \chi(nt),
\]
which shows that $\wt{\chi}(t)$ does not depend on the choice of $n$.
\end{proof}

\begin{remark}
Contrary to \cite[Proposition 3.5.6]{GT06}, we use that semicharacters of $\S$ extend uniquely to the seminormalization of $\S$ rather than to the \textit{normalization} $\S_{\textup{n}}: = \{g \in \G \mid \exists n \in \bN \textup{ such that } ng \in \S\}$.
\end{remark}

\begin{remark}
It is worth noticing that the equivalence of items (ii) and (iii) of Proposition \ref{P:inj theta} can follow also by the universal property of seminormalizations, by applying \cite[Lemma 1.11]{CHWW14} for the pointed monoid $\ol{\bD}$.
Therein the existence of the map $\iota^*$ follows by applying a Zorn's Lemma.

However it is the analytic form of $\iota^*$ that we will be requiring and wish to make explicit here.
Suppose that $\S \subset \bZ_+^d$ is a positive cone with $\S_{\sn} = \bZ_+^d$.
For every $\Bi \in \bZ_+^d$ let $n_i \in \bN$ such that both $(n_i +1) \Bi$ and $n_i \Bi$ are in $\S$.
Proposition \ref{P:inj theta} asserts that, if $\chi \in \fM_\S$ with $\chi = \ev_\zeta|_{\A(\S)}$, then $\zeta$ is uniquely given by
\begin{equation}\label{eq:analytic}
\zeta_i = 
\begin{cases} 
\chi(z_i^{n_i +1}) / \chi(z_i^{n_i}) & \text{ if } \chi(z_i^{n_i}) \neq 0, \\ 
0 & \text{ if } \chi(z_i^{n_i}) = 0. 
\end{cases}
\end{equation}
\end{remark}

Recall that if $\S_1 \simeq \S_2$, and they are both positive cones of $\bZ$ then $\S_1 = \S_2$.
This property passes also to higher ranks.

\begin{proposition}\label{P:iso class}
Let $\S_1 \subset \bZ^{d_1}$ and $\S_2 \subset \bZ^{d_2}$ be positive cones such that $(\S_1)_{\sn} = \bZ_+^{d_1}$ and $(\S_2)_{\sn} = \bZ_+^{d_2}$.
Then $\S_1 \simeq \S_2$ if and only if $d_1 = d_2$ and $\S_1 = \S_2$ up to a permutation of the coordinates.
\end{proposition}

\begin{proof}
Let $\si \colon \S_1 \to \S_2$ be a semigroup isomorphism.
Since it defines an isomorphism between the groups generated by $\S_1$ and $\S_2$, we get that $d_1 = d_2$, which we name as $d$ from now on.
Moreover the induced group isomorphism is given by a unitary map, say $U \in \GL_d(\bZ)$.
For every $i \in \{1, \dots, d\}$, choose $n_i \in \bN$ so that $n_i \Bi \in \S_1(i)$.
Then the $i$-th column $\si(n_i \Bi)$ of $n_i U$ is in $\S_2 \subset \bZ_+^d$ and so it has non-negative entries.
Hence all entries of $U$ are non-negative integers.
As the same holds for $U^{-1}$ we get that $U$ is a permutation matrix.
\end{proof}

We now have arrived to the main rigidity result.
We will be using an idea of \cite{DRS11} for rotating isomorphisms to vacuum preserving isomorphisms.

\begin{theorem}\label{T:hr num}
Let $\S_1 \subset \G_1$ and $\S_2 \subset \G_2$ be higher rank numerical semigroups.
Then the following are equivalent:
\begin{enumerate}
\item $\S_1 \simeq \S_2$;
\item $\A(\S_1) \simeq \A(\S_2)$ by a completely isometric isomorphism;
\item $\A(\S_1) \simeq \A(\S_2)$ by an algebraic isomorphism.
\end{enumerate}
\end{theorem}

\begin{proof}
First we remark that semigroup isomorphisms induce completely isometric isomorphisms.
Indeed for an isomorphism $\si \colon \S_1 \to \S_2$ we can define $U \colon \ell^2(\S_1) \to \ell^2(\S_2)$ to be the permutation unitary $U e_t = e_{\si(t)}$.
It then follows that $U V_s U^* = V_{\si(s)}$ for all $s \in \S_1$.
Therefore it suffices to show that item (iii) implies item (i).
Recall that the isomorphism in item (iii) is automatically bounded by Proposition \ref{P:aut cont}.

Combining the above with Remark \ref{R:rem}, we may assume without loss of generality that $\S_1 \subset \bZ_+^{d_1}$ with $(\S_1)_{\sn} = \bZ_+^{d_1}$, and likewise for $\S_2$.
Thus by Proposition \ref{P:iso class} it suffices to show that item (iii) implies that $d_1 = d_2$ and that $\S_1 = \S_2$ up to a permutation of the variables.

An algebraic isomorphism $\rho$ between $\A(\S_1)$ and $\A(S_2)$ implements a homeomorphism $\rho^*$ of their character spaces $\ol{\bD}^{d_1}$ and $\ol{\bD}^{d_2}$.
Therefore $d_1 = d_2$, which we name as $d$ henceforth.
For convenience we will treat the cases $d=1$ and $d>1$ separately.

\smallskip

\noindent
\textbf{The one-variable case.}
For $d=1$ we have $\fM_{\S_1} \simeq \fM_{\S_2} \simeq \ol{\bD}$.
We employ a technique from \cite{DRS11} to rotate the isomorphism to one that matches the zeroes of the character space.
To this end, for $\vartheta \in \bR$ let the rotation map
\[
\rho_\vartheta = \ad_{u_\vartheta} \textup{ with } u_{\vartheta} \colon \ell^2(\bZ_+) \to \ell^2(\bZ_+) : e_n \mapsto e^{i\vartheta n} e_n.
\]
It is immediate that $\rho_\vartheta$ gives an automorphism of $\A(\bZ_+)$ that sends every generator to a scalar multiple of the same generator.
Hence the restriction to $\A(\S_1)$ and to $\A(\S_2)$ gives complete isometric automorphisms such that
\[
(\rho_\vartheta)^*(\ze) = e^{i \vartheta} \ze \quad \foral \ze \in \bD.
\]
Suppose that $\ze := \rho^*(0) \neq 0$ and so $ \eta:= (\rho^{-1})^*(0) \neq 0$.
Then
\[
C_1 := \{\rho_\vartheta^*(\zeta) \mid \vartheta \in [0,2\pi]\}
\]
defines a circle of radius $|\zeta|$ around the origin in $\fM_{\S_1}$.
Likewise we can form a circle 
\[
C_2 := \{\rho_{\vartheta}^*(\eta) \mid \vartheta \in [0,2\pi]\}
\]
of radius $|\eta|$ around the origin in $\fM_{\S_2}$.
By applying $\rho^*$ we can implement a closed curve
\[
\rho^*(C_2) = \{(\rho_\vartheta \circ \rho)^*(\eta) \mid \vartheta \in [0,2\pi]\}
\]
that has $\zeta$ in its interior and passes through $0 \in \fM_{S_1}$.
Since $\fM_{\S_1} \simeq \ol{\bD}$ there exists a point of intersection $\ze' \neq \ze$ between $\rho^*(C_2)$ and $C_1$.
Then $\eta':= (\rho^{-1})^*(\ze') \neq \eta$, and $\eta'$ lies on the circle $C_2$.
Choose $\vartheta_1$ that rotates $\ze$ to $\ze'$, and $\vartheta_2$ that rotates $\eta'$ to $\eta$.
Then we can define the isomorphism
\[
\wh{\rho}:= \rho \circ \rho_{\vartheta_1} \circ \rho^{-1} \circ \rho_{\vartheta_2} \circ \rho \colon \A(\S_1) \to \A(\S_2),
\]
for which
\[
(\wh{\rho})^*(0)
=
\rho^* \circ \rho_{\vartheta_2}^* \circ (\rho^{-1})^* \circ \rho_{\vartheta_1}^* \circ \rho^*(0)
=
\rho^* \circ \rho_{\vartheta_2}^*(\eta')
=
\rho^*(\eta) = 0.
\]
This transformation is depicted in the following figure.

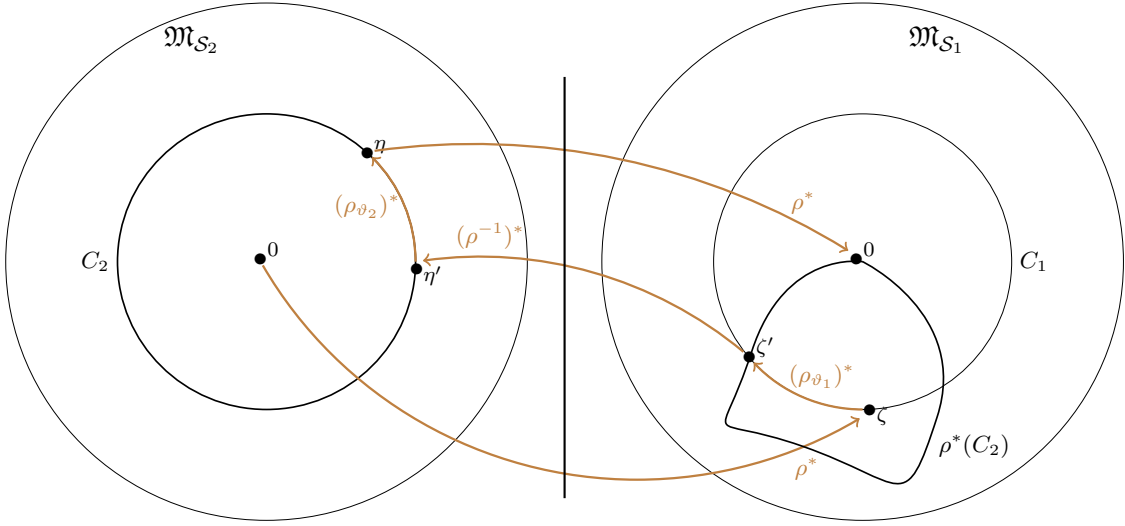
\begin{figure}[h]
\begin{tikzpicture}[thick,scale=.98, every node/.style={transform shape}]
\node (A) at (-4,.1) {$\bullet^0$};
\node (B) at (-2.56,1.51) {$\bullet^\eta$};
\node (C) at (-1.85,-.17) {$\bullet_{\eta'}$};
\draw [line width=.22mm] (-4,0) circle (2cm);
\draw [line width=.05mm] (-4,0) circle (3.5cm);

\node (D) at (4,.1) {$\bullet^0$};
\node (E) at (4.18,-2.07) {$\bullet_\zeta$};
\node (F) at (2.61,-1.2) {$\bullet^{\zeta'}$};
\draw [line width=.1mm] (4,0) circle (2cm);
\draw [line width=.05mm] (4,0) circle (3.5cm);

\draw [->, line width=.3mm, brown] (-2,0) arc (0:45:2cm);
\draw [->, line width=.3mm, brown] (4,-2) arc (270:223:2cm);
\draw [->, line width=.3mm, brown] (-4.05,-.05) arc (210:301:5.8cm);
\draw [->, line width=.3mm, brown] (2.4,-1.22) arc (50:98:5.5cm);
\draw [->, line width=.3mm, brown] (-2.55,1.5) arc (98:58:9.5cm);

\draw [line width=.22mm] plot [smooth, tension=2] coordinates {(4,0) (2.51,-1.22) (3.2,-2.5) (5,-2.1) (4,0)};

\node (G) at (-6.3,0) {{\footnotesize $C_2$}};
\node (H) at (6.3,0) {{\footnotesize $C_1$}};
\node (J) at (5.5,-2.5) {{\footnotesize $\rho^*(C_2)$}};

\node (K) at (-5,3) {$\mathfrak{M}_{\mathcal{S}_2}$};
\node (L) at (5,3) {$\mathfrak{M}_{\mathcal{S}_1}$};

\draw (0,2.5) -- (0,-3.2);

\node (M) at (3.45,-1.55) {\textcolor{brown}{{\footnotesize $(\rho_{\vartheta_1})^*$}}};
\node (N) at (-2.65,.75) {\textcolor{brown}{{\footnotesize $(\rho_{\vartheta_2})^*$}}};
\node (O) at (3.25,-2.8) {\textcolor{brown}{{\footnotesize $\rho^*$}}};
\node (O) at (-1,.35) {\textcolor{brown}{{\footnotesize $(\rho^{-1})^*$}}};
\node (O) at (3.2,.8) {\textcolor{brown}{{\footnotesize $\rho^*$}}};

\end{tikzpicture}


\caption*{{\small Figure. Matching zeroes of the character spaces.}}
\end{figure}

Hence without loss of generality we may assume that $\A(\S_1) \simeq \A(\S_2)$ by an isomorphism $\rho$ such that $\rho^*(0) = 0$.
Now we use the explicit construction of equation (\ref{eq:analytic}).
Recall here that we identify elements in $\A(\S)$ with their corresponding holomorphic functions.
Fix $0 \neq n \in \S_1$ such that $n + 1 \in \S_1$ and set
\[
f := \rho(z^{n + 1}) \qand g := \rho(z^n).
\]
Then $\rho^*(\ze) = f(\ze)/g(\ze)$ whenever $g(\ze) \neq 0$.
However $f/g$ is holomorphic in $\bD \setminus g^{-1}(\{0\})$ and continuously extendable at any $w \in g^{-1}(\{0\})$ by $\rho^*(w)$. 
By Riemann's Theorem on removable singularities $f/g$ is holomorphically extendable to $\bD$, and thus its extension $\rho^*$ is holomorphic on $\bD$.

Clearly $\rho^*$ is not constant.
Thus by the open mapping theorem for holomorphic functions we have that $\rho^*(\bD) \subset \bD$.
By symmetry we have the same for its inverse.
Hence $\rho^*$ is a biholomorphism of $\bD$ with $\rho^*(0) = 0$.
Thus by Schwarz Lemma it follows that $\rho^*(\zeta) = e^{i\vartheta} \zeta$ for $\vartheta \in [0,2\pi]$.
Hence we get that $(\rho_{-\vartheta} \circ \rho)^* = \id$.
As rotations are automorphisms we may work with $\rho_{-\vartheta} \circ \rho$ instead of $\rho$.
Thus without loss of generality we may assume that $\rho^* = \id$ on $\bD$; and hence on $\ol{\bD}$.

To finish the first part, let $s \in \S_1$ and write $\rho(z^s) = h(z)$.
Then for every $\ol{\bD} \ni \zeta \equiv \ev_\ze \in \fM_{\S_2}$ we have that $\rho^*(\ev_\ze) = \ev_\ze \in \fM_{\S_1}$ and so
\begin{equation}\label{eq:in}
\ze^s = \ev_\ze(z^s) = \rho^*(\ev_\ze)(z^s) = \ev_\ze(\rho(z^s)) = h(\ze).
\end{equation}
As this holds for all $\ze \in \ol{\bD}$ we derive that $z^s = h(z) = \rho(z^s)$.
Since $s \in \S_1$ was arbitrary we get that $\rho = \id|_{\S_1}$ giving that $\S_1 \subset \S_2$.
By symmetry on $\rho^{-1}$ we have equality.

\smallskip

\noindent
\textbf{The multi-variable case.}
Let the continuous functions $\rho^*_i \colon \ol{\bD}^d \to \ol{\bD}$ so that the homeomorphism $\rho^* \colon \fM_{\S_2} \to \fM_{\S_1}$ is written as
\[
\rho^*(\ze) = (\rho^*_1(\ze), \dots, \rho^*_d(\ze)).
\]
First we show that every $\rho^*_i$ is holomorphic on $\bD^d$.
Fix $n_i \in \bN$ so that $n_i \bo{i}$ and $(n_i + 1) \bo{i}$ are both in $\S_1(i)$.
Set $f_i := \rho(z_i^{n_i + 1})$ and $g_i := \rho(z_i^{n_i})$.
By using equation (\ref{eq:analytic}) we can write
\[
\rho^*_i(\ze) = \frac{f_i(\ze)}{g_i(\ze)} \foral \ze \in \bD^d \setminus g_i^{-1}(\{0\}).
\]
Zero sets of analytic functions are thin sets, and by \cite[Theorem 3.4]{Ran86} the set $g_i^{-1}(\{0\})$ can be removed so that $\rho^*_i$ is holomorphic on $\bD^d$.

Now we have that $\rho^*_1$ cannot be constant as in that case we would have the contradiction $\ol{\bD}^d = \rho^*(\ol{\bD}^d) \subset \{\rho^*_1(0, \dots, 0)\} \times \ol{\bD}^{d-1}$.
By applying the open mapping theorem for holomorphic functions on several variables, e.g. \cite[Theorem 1.21]{Ran86}, we get that $\rho^*_1(\bD^d) \subset \bD$.
Likewise it follows that
\[
\rho^*_i(\bD^d) \subset \bD \foral i =1, \dots, d.
\]
Hence we have that $\rho^*(\bD^d) \subset \bD^d$, and thus by symmetry on $(\rho^*)^{-1}$, we conclude that $\rho^*(\bD^d) = \bD^d$.
Therefore $\rho^*$ restricts to a biholomorphism of the polydisc.
Recall that 
\[
\Aut(\bD^d) \simeq (\times_{i=1}^d \Aut(\bD)) \rtimes \S_d,
\]
e.g. \cite[Theorem 2, pp. 48]{Sha92}.
Hence $\rho^*$ is the product of automorphisms of $\bD$ up to a permutation of the variables, say $\si$.
As every permutation on the variables implies a completely isometric isomorphism, without loss of generality, we may substitute $S_2$ by $\si(S_2)$ so that the $\rho^*_i$ depends only on $\zeta_i$.

Note that rotating coordinate-wise on $\A(\bZ^d_+)$ restricts to automorphisms on $\A(\S_1)$ and on $\A(\S_2)$.
Therefore by applying the one-variable arguments we can rotate $\rho$ appropriately and coordinate-wise so that $\rho^*_i(0) = 0$ for all $i = 1, \dots, d$.
That is, every $\rho^*_i$ restricts to a biholomorphism of $\bD$ fixing the zero and so every $\rho^*_i|_{\bD}$ is a rotation.
Hence without loss of generality $\rho^* = \id$ on $\bD^d$ and thus on $\ol{\bD}^d$.
A computation as in equation (\ref{eq:in}) shows that $\rho = \id$ and the proof is complete.
\end{proof}

We isolate the following corollary from the proof of Theorem \ref{T:hr num}.

\begin{corollary}
Let $\S_1 \subset \bZ^{d}$ and $\S_2 \subset \bZ^{d}$ be higher rank numerical semigroups.
Then an algebraic isomorphism between $\A(\S_1)$ and $\A(\S_2)$ is vacuum preserving if and only if it is the composition of a permutation of co-ordinates by a rotation.
\end{corollary}


\end{document}